\documentclass[10pt,leqno]{article} %%%%leqno es para numeraci\'{o}n de
   \usepackage[centertags]{amsmath}
   \usepackage{amsfonts}
   \usepackage{amsmath}
   \usepackage{amssymb}
   \usepackage{amsthm}
   \usepackage{newlfont}
   \usepackage[latin1]{inputenc}%%%%para que reconozca las tildes.

 \usepackage{multirow, bigdelim}

\allowdisplaybreaks[3]
\theoremstyle{plain}
\newtheorem{theorem}{Theorem}[section]
\newtheorem{lemma}[theorem]{Lemma}

\newtheorem{corollary}[theorem]{Corollary}

\theoremstyle{definition}
\newtheorem{definition}[theorem]{Definition}

\theoremstyle{remark}

\newtheorem*{remark*}{Remark}

\numberwithin{equation}{section}

\newcommand{\dosfilas}[2]{
  \ldelim[{2}{2mm}& #1 &\rdelim]{2}{2mm} \\
  & #2 & &  & &
}

\catcode`,\active

\catcode`\,12

\newcommand\D{{\mathcal D}}

\newcommand\F{{\mathcal F}}
\newcommand\G{{\mathcal G}}

\newcommand\RR{{\mathbb R}}
\newcommand\ZZ{{\mathbb Z}}
\newcommand\NN{{\mathbb N}}

\newcommand\Sh{\mbox{\Large $\mathfrak {s}$}}

   \title{Higher order recurrence relation for exceptional Charlier, Meixner, Hermite and Laguerre orthogonal polynomials
  \footnote{Partially supported by MTM2012-36732-C03-03 (Ministerio de Economía y Competitividad),
FQM-262, FQM-4643, FQM-7276 (Junta de Andalucía) and Feder Funds (European
Union).}}
   \author{Antonio J. Dur\'{a}n\\
     \footnotesize
        \  Departamento de An\'{a}lisis Matem\'{a}tico.
       Universidad de Sevilla \\
       \footnotesize Apdo (P. O. BOX) 1160. 41080 Sevilla. Spain.
   duran@us.es \\
          \ \ }
   \date{}
   
\begin{document}
   \maketitle

\bigskip

\begin{abstract}
In this paper we prove in a constructing way that exceptional Charlier, Meixner, Hermite and Laguerre polynomials satisfy  higher order recurrence relations. Our conjecture is that the recurrence relations provided in this paper have minimal order.
\end{abstract}

\section{Introduction}
Exceptional and exceptional discrete orthogonal polynomials $p_n$, $n\in X\varsubsetneq \NN$, are complete orthogonal polynomial systems with respect to a positive measure which in addition are eigenfunctions of a second order differential or difference operator, respectively. They extend the classical families of Hermite, Laguerre and Jacobi, and the classical discrete families of Charlier, Meixner, Krawtchouk and Hahn. The last few years have seen a great deal of activity in the area of exceptional and exceptional discrete orthogonal polynomials (see, for instance,
\cite{DEK,duch,dume,GUKM1}, \cite{GUKM2} (where the adjective \textrm{exceptional} for this topic was introduced), \cite{GUKM3,GUKM4,GUGM,G,GQ,MR,OS0,OS,OS3,Qu,STZ,Ta} and the references therein). One can also add to the list exceptional discrete polynomials on nonuniform lattices and exceptional $q$-orthogonal polynomials related to second order $q$-difference operators (\cite{OS,OS2,OS4,OS5,OS6}).

The most apparent difference between classical or classical discrete orthogonal polynomials and their exceptional counterparts
is that the exceptional families have gaps in their degrees, in the
sense that not all degrees are present in the sequence of polynomials (as it happens with the classical families) although they form a complete orthonormal set of the underlying $L^2$ space defined by the orthogonalizing positive measure. This
means in particular that they are not covered by the hypotheses of Bochner's and Lancaster's classification theorems (see \cite{B} or \cite{La}) for classical and classical discrete orthogonal polynomials, respectively.
Exceptional orthogonal polynomials have been applied to shape-invariant potentials \cite{Qu},
supersymmetric transformations \cite{GUKM3}, to discrete quantum mechanics \cite{OS}, mass-dependent potentials \cite{MR}, and to quasi-exact solvability \cite{Ta}.

Favard's Theorem establishes that a sequence of polynomials $(p_n)_{n\in \NN}$, $p_n$ of degree $n$, is orthogonal (with non null norm) with respect to a measure supported in the real line if and only if it satisfies a three term recurrence relation of the form ($p_{-1}=0$)
$$
xp_n(x)=a_np_{n+1}(x)+b_np_n(x)+c_np_{n-1}(x), \quad n\ge 0,
$$
where $(a_n)_{n\in \NN}$, $(b_n)_{n\in \NN}$ and $(c_n)_{n\in \NN}$ are sequences of real numbers with $a_{n-1}c_n\not =0$, $n\ge 1$. If, in addition, $a_{n-1}c_n>0$, $n\ge 1$,
then the polynomials $(p_n)_{n\in \NN}$ are orthogonal with respect to a positive measure with infinitely many points in its support, and conversely.

The gaps in their degrees imply that exceptional orthogonal polynomials do not satisfy three term recurrence relations as the usual orthogonal polynomials do. However, as we point out in \cite{duch}, these  families of exceptional polynomials satisfy higher order recurrence relations of the form
\begin{equation}\label{horr}
\lambda (x) p_n(x)=\sum_{j=-r}^ra_{n,j}p_{n+j},\quad n\ge n_0,
\end{equation}
where $\lambda$ is a polynomial of degree $r$, $(a_{n,j})_{n}$, $j=-r,\cdots ,r$, are sequences of numbers independent of $x$ (called recurrence coefficients), with $a_{n,r}\not =0$, for $n$ big enough and $n_0$ is certain nonnegative integer. We say that this high order recurrence relation has order $2r+1$.
Some examples of these higher order recurrence relations already appeared in \cite{STZ}. For other kind of higher order recurrence relations with recurrence coefficients depending on $x$ satisfied by exceptional polynomials see \cite{od} and \cite{GUGM}. We say that a recurrence relation $T$ of the form (\ref{horr}) is of minimal order if any other recurrence relation $\tilde T$ of the form (\ref{horr}) satisfied by the polynomials $(p_n)_n$ has order bigger that or equal to $T$.

The purpose of this paper is to prove that exceptional Charlier, Meixner, Hermite and Laguerre polynomials always satisfy higher order recurrence relations of the form (\ref{horr}). We also provide a method to explicitly find the recurrence coefficients. We conjecture that our method provides the minimal order recurrence relation. For example the recurrence relations considered in \cite{STZ} for some instances of exceptional Laguerre polynomials have order $4\ell+1$, where $\ell$ is certain nonnegative integer, while the ones we provide in this paper for the same exceptional polynomials have order $2\ell+3$.

In \cite{duch} and \cite{dume}, we have constructed exceptional Charlier and Meixner polynomials from Krall discrete polynomials by using the
concept of dual families of polynomials (see \cite{Leo}).

\begin{definition}\label{dfp}
Given two sets of nonnegative integers $U,V\subset \NN$, we say that the two sequences of polynomials
$(p_u)_{u\in U}$, $(q_v)_{v\in V}$ are dual if there exist a couple of sequences of numbers $(\xi_u)_{u\in U}, (\zeta_v)_{v\in V} $ such that
\begin{equation}\label{defdp}
\xi_up_u(v)=\zeta_vq_v(u), \quad u\in U, v\in V.
\end{equation}
\end{definition}

It turns out that duality interchanges exceptional discrete orthogonal polynomials with the so-called Krall discrete orthogonal polynomials. A Krall discrete orthogonal family is a sequence of polynomials $(q_n)_{n\in \NN}$, $q_n$ of degree $n$, orthogonal with respect to a positive measure which, in addition, are also eigenfunctions of a higher order difference operator. A huge amount of families of Krall discrete orthogonal polynomials have been recently introduced by the author by mean of certain Christoffel transform of the classical discrete measures of Charlier, Meixner, Krawtchouk and Hahn (see \cite{du0,du1,dudha,DdI}). A Christoffel transform is a transformation which consists in multiplying a measure $\mu$ by a polynomial $r$. It has a long tradition in the context of orthogonal polynomials: it goes back a century and a half ago when E.B. Christoffel (see \cite{Chr} and also \cite{Sz}) studied it for the particular case $r(x)=x$.

Our procedure to construct the higher order recurrence relations for the exceptional discrete polynomials consists in applying duality to the higher order difference operator with respect to which the associated Krall discrete polynomials are eigenfunctions. This will be done in Sections 2 and 4 for exceptional Charlier and Meixner polynomials, respectively.

One can  obtain  exceptional Hermite and Laguerre polynomials by taking limits in some of the parameters of the exceptional Charlier and Meixner polynomials, respectively. This can be done in the same way as one goes from Charlier and Meixner polynomials to Hermite and Laguerre polynomials, respectively, in the Askey tableau. By taking limit in the higher order recurrences relation for the exceptional Charlier and Meixner polynomials, one can also find higher order recurrence relations for exceptional Hermite and Laguerre polynomials. This will be done in Sections 3 and 5, respectively.

Recurrence relation for exceptional Hahn and Jacobi polynomials will be provided in the forthcoming \cite{duha}.

\section{Exceptional Charlier polynomials}
We start with some basic definitions and facts about Charlier  polynomials.

For $a\neq0$, we write $(c_n^a)_n$ for the sequence of Charlier polynomials (the next formulas can be found in \cite{Ch}, pp. 170-1; see also \cite{KLS}, pp., 247-9 or \cite{NSU}, ch. 2) defined by
\begin{equation}\label{Chpol}
    c_n^a(x)=\frac{1}{n!}\sum_{j=0}^n(-a)^{n-j}\binom{n}{j}\binom{x}{j}j!.
\end{equation}
For $n<0$, we write $c_n^a=0$.
The Charlier polynomials are orthogonal with respect to the measure
\begin{equation}\label{Chw}
    \rho_a=\sum_{x=0}^{\infty}\frac{a^x}{x!}\delta_x,\quad a\neq0,
\end{equation}
which is positive only when $a>0$.

They are eigenfunctions of the following second-order difference operator
\begin{equation}\label{Chdeq}
   D_a=-x\Sh_{-1}+(x+a)\Sh_0-a\Sh_1,\quad D_a(c_n^a)=nc_n^a,\quad n\geq0,
\end{equation}
where $\Sh_j(f)=f(x+j)$.

\bigskip
From now on, $F$ will denote a finite set of positive integers. We will write $F=\{ f_1,\cdots , f_k\}$, with $f_i<f_{i+1}$. Hence $k$ is the number of elements of $F$ and $f_k$ is the maximum element of $F$.

We associate to $F$ the nonnegative integers $u_F$ and $w_F$ and the infinite set of nonnegative integers $\sigma_F$ defined by
\begin{align}\label{defuf}
u_F&=\sum_{f\in F}f-\binom{k+1}{2},\\\label{defwf}
w_F&=\sum_{f\in F}f-\binom{k}{2}+1,\\\label{defsf}
\sigma _F&=\{u_F,u_F+1,u_F+2,\cdots \}\setminus \{u_F+f,f\in F\}.
\end{align}
The infinite set $\sigma_F$ will be the set of indices for the exceptional Charlier or Hermite polynomials associated to $F$.

Along this paper, we use the following notation: given a finite set of positive integers $F=\{f_1,\ldots , f_k\}$, the expression
\begin{equation}\label{defdosf}
  \begin{array}{@{}c@{}cccc@{}c@{}}
    \dosfilas{ z_{f,1} & z_{f,2} &\cdots  & z_{f,k} }{f\in F}
  \end{array}
\end{equation}
inside of a matrix or a determinant will mean the submatrix defined by
$$
\left(
\begin{array}{cccc}
z_{f_1,1} & z_{f_1,2} &\cdots  & z_{f_1,k}\\
\vdots &\vdots &\ddots &\vdots \\
z_{f_k,1} & z_{f_k,2} &\cdots  & z_{f_k,k}
\end{array}
\right) .
$$

We are now ready to introduce exceptional Charlier polynomials (see \cite{duch}).

\begin{definition}
For a given real number $a\not =0$ and a finite set $F$ of positive integers, we define the polynomials $c_n^{a;F}$, $n\ge 0$, as
\begin{equation}\label{defchex}
c_n^{a;F}(x)=\left|
  \begin{array}{@{}c@{}cccc@{}c@{}}
    & c_{n-u_F}^a(x)&c_{n-u_F}^a(x+1)&\cdots &c_{n-u_F}^a(x+k)& \\
    \dosfilas{c_{f}^a(x)  & c_{f}^a(x+1) &\cdots  & c_{f}^a(x+k)}{f\in F}
  \end{array}
  \right| ,
\end{equation}
where the number $u_F$ is defined by (\ref{defuf}) (the determinant (\ref{defchex}) should be understood as explained above: see (\ref{defdosf})).
\end{definition}
We have that for $n\in \sigma _F$ (see  (\ref{defsf})), $c_n^{a;F}$ is a polynomial of degree $n$.  But for $n\not \in \sigma_F$ the determinant (\ref{defchex}) vanishes and then $c_n^{a;F}=0$ (if $n<u_F$, the first row is zero, and for
$n\ge u_F$ and $n\not \in \sigma_F$, there are two equal rows).

In \cite{duch}, we proved that these polynomials are always eigenfunctions of a second order difference operator with rational coefficients. Under the assumption $a>0$ and the admissibility condition
\begin{equation}\label{admch}
\prod_{f\in F}(x-f)\ge 0,\quad x\in \NN,
\end{equation}
the polynomials $(c_n^{a;F})_{n\in \sigma_F}$ are orthogonal (and complete) with respect to a positive measure (see Theorems 4.4 and 4.5 in \cite{duch}). We call these polynomials exceptional Charlier polynomials. Since we want to work with orthogonal polynomials with respect to positive measures we will assume
that $a>0$ and that the admissibility condition (\ref{admch}) holds, although these assumptions are not needed for the implementation of our method to find higher order recurrence relations for the polynomials (\ref{defchex}).

Related to the exceptional Charlier polynomials is the
Casoratian determinant defined by
\begin{equation}\label{casc}
\Omega _F ^a (x)=\det (c_{f}^{a}(x-j+1))_{i,j=1}^k.
\end{equation}
$\Omega _F^a$ is a polynomial of degree $w_F-1$ (see (\ref{defwf}) for the definition of the number $w_F$), which enjoys the following nice symmetry
$$
\Omega _F ^a (x)=(-1)^{u_F+k}\Omega _{I(F)}^{-a}(-x),
$$
where $I$ is the involution defined in the set $\Upsilon$  formed by all finite sets of positive  integers by
\begin{align}\label{dinv}
I(F)=\{1,2,\cdots, f_k\}\setminus \{f_k-f,f\in F\},
\end{align}
(see \cite{duch}, (3.28)).

Up to an additive constant, we define the polynomial $\lambda _F^a$  of degree $w_F$ by solving the
first order difference equation
\begin{equation}\label{lch}
\lambda_F^a (x)-\lambda_F^a (x-1)=\Omega _F^a(x).
\end{equation}
As we will see below, the higher order recurrence relation for the exceptional Charlier polynomials is constructed from this polynomial $\lambda_F^a$.

Consider now the measure
\begin{equation}\label{mraf}
\rho _{a}^{F}=\sum _{x=u_F}^\infty \prod_{f\in F}(x-f-u_F)\frac{a^{x-u_F}}{(x-u_F)!}\delta _x.
\end{equation}
For $a>0$ and under the assumption (\ref{admch}), this measure is positive.
Notice that the measure $\rho_{a}^{F}$ is supported in the infinite set of nonnegative integers  $\sigma_F$ (\ref{defsf}).

The measure $\rho_a^\F$ has associated a sequence of orthogonal polynomials  $q_n^{a;F}$, $n\ge 0$, which can be constructed using the Christoffel-Szeg\"o determinantal formula (\cite{Sz}, Th. 2.5)
\begin{equation}\label{defqnch}
q_n^{a;F}(x)=\frac{\left|
  \begin{array}{@{}c@{}cccc@{}c@{}}
    & c_n^a(x-u_F)&c_{n+1}^a(x-u_F)&\cdots &c_{n+k}^a(x-u_F)& \\
    \dosfilas{c_n^a(f)&c_{n+1}^a(f)&\cdots &c_{n+k}^a(f)}{f\in F}
  \end{array}
  \right| }{\prod_{f\in F}(x-f-u_F)}.
\end{equation}
In \cite{DdI}, Th. 1.1, it is proved that the polynomials $q_n^{a;F}(x+u_F)$, $n\ge 0$, are eigenfunctions of a higher order difference operator. This operator can be explicitly constructed by means of the formula
\begin{equation}\label{expexp}
D_F=\lambda_F^a(D_a)+\sum_{g\in I(F)}M_g(D_a)\circ \nabla \circ c_g^{-a}(-D_a-1)
\end{equation}
where $I$ is the involution (\ref{dinv}), $\nabla$ is the first order operator $\nabla f(x)=f(x)-f(x-1)$, $D_a$ is the Charlier second order difference operator (\ref{Chdeq}), $\lambda _F^a$ is the polynomial (\ref{lch}) and $M_h$, $h=1,\cdots , m$, are certain polynomials which can be explicitly constructed (see \cite{DdI}, Section 5). The associated eigenvalues are given
by the polynomial $\lambda _F^a$, so that $D_F(q_n^{a;F})=\lambda _F^a(n)q_n^{a;F}$.

In \cite{DdI}, Theorem 1.1, it is also proved that $D_F$ is a difference operator of order $2w_F+1$, which can be written in terms of the shift operators $\Sh_j$ in the form
\begin{equation}\label{expexp2}
D_F=\sum_{j=-w_F}^{w_F} h_j(x)\Sh _j,
\end{equation}
where $h_j$, $j=-w_F,\cdots, w_F$, are certain polynomials.

It turns out that the exceptional Charlier polynomials $c_n^{a;F}$, $n\in \sigma_F$, are strongly related by duality with the polynomials $q_n^{a;F}$, $n\ge 0$.

\begin{lemma}[Lemma 3.2 of \cite{duch}]\label{lem3.2}
If $u$ is a nonnegative integer and $v\in \sigma_F$, then
\begin{equation}\label{duaqnrn}
q_u^{a;F}(v)=\xi_u\zeta_vc_v^{a;F}(u),
\end{equation}
where
\begin{equation}\label{zi}
\xi_u=\frac{(-a)^{(k+1)u}}{\prod_{i=0}^k(u+i)!},\quad \zeta_v=\frac{(-a)^{-v}(v-u_F)!\prod_{f\in F}f!}{\prod_{f\in F}(v-f-u_F)}.
\end{equation}
\end{lemma}

We are now ready to establish the main result of this section.

\begin{corollary}\label{cor1} The exceptional Charlier polynomials defined by (\ref{defchex}) satisfy a $2w_F+1$ order recurrence relation of the form
\begin{equation}\label{horrchex}
\sum_{j=-w_F}^{w_F}A_j^{a;F}(n)c_{n+j}^{a;F}(x)=\lambda _F^{a} (x) c_{n}^{a;F}(x),\quad n\ge 0,
\end{equation}
where the number $w_F$ and the polynomial $\lambda_F^a$ are defined by (\ref{defwf}) and (\ref{lch}), respectively. For $j=-w_F,\cdots , w_F$, $A_j^{a;F}(n)$ is a rational function in $n$ which does not depend on $x$ (and whose denominator does not vanish for $n\in \NN$).
\end{corollary}

\begin{proof}
The starting point is the eigenvalue equation for the difference operator $D_F$ and the polynomials $(q_m^{a;F})_{m\in \NN}$
\begin{equation}\label{expexpt}
\sum_{j=-w_F}^{w_F} h_j(n)q_m^{a;F}(n+j)=\lambda_F^a(m)q_m^{a;F}(n).
\end{equation}
When $n+j\in \sigma_F$, $j=-w_F,\cdots , w_F$, using the duality (\ref{duaqnrn}) we get
\begin{equation}\label{expexps}
\sum_{j=-w_F}^{w_F} h_j(n)\xi_m\zeta_{n+j}c_{n+j}^{a;F}(m)=\lambda_F^a(m)\xi_m\zeta_{n}c_n^{a;F}(m).
\end{equation}
Since $\xi_m\not =0$ and $\zeta_n\not=0$, $n\in \sigma_F$, we deduce (\ref{horrchex}) where
\begin{equation}\label{defaj}
A_j^{a;F}(n)=\frac{h_j(n)\zeta_{n+j}}{\zeta_n}.
\end{equation}
When $n\not \in \sigma_F$, we have from the definition \ref{defchex} that $c_n^{a;F}=0$ and hence
the right hand side of (\ref{horrchex}) vanishes. We now see that the left hand side of (\ref{horrchex}) also vanishes. Indeed,
if $n+j\not \in \sigma_F$, again by the definition (\ref{defchex}) we have $c_{n+j}^{a;F}=0$. If $n+j\in \sigma_F$, then by definition of $\zeta_{n}$ (see (\ref{zi})) we have that $\zeta_{n+j}/\zeta_n=0$ and then $A^{a;F}_{n+j}=0$ (see (\ref{defaj})). In any case all the addends in the left hand side of (\ref{horrchex}) vanish.

The case when $n\in \sigma_F$ and $n+j\not\in \sigma_F$ for some $j$, $j=-w_F,\cdots , w_F$, is the more difficult one. The proof can be sketched as follows. Since the orthogonal polynomials with respect to the positive measure $\rho_a^F$ are eigenfunctions of the higher order difference operator $D_F$, we can conclude that $D_F$ is symmetric with respect to $\rho_a^F$ (i.e.,
for all polynomials $p$, $q$,
$$
\sum _{x=u_F}^\infty (D_Fp)(x)q(x)\rho_a^F(x)=\sum _{x=u_F}^\infty p(x)(D_Fq)(x)\rho_a^F(x)).
$$
It can be proved that the symmetry of a difference operator with respect to a discrete  weight is characterized by certain difference equations and a set of boundary conditions (see Theorem 3.2 of \cite{du0}). Since the support of $\rho _a^F$ is $\sigma_F$, for the difference operator $D_F$ with coefficients $h_j$, $j=-w_F,\cdots, w_F$, these boundary conditions are (see \cite{du0}, (3.2) and (3.4))
\begin{align*}
h_j(x-j)&=0\qquad \mbox{for $x\in (j+\sigma_F)\setminus \sigma_F$ and $j=1,\cdots, w_F$,}\\
h_{-j}(x)&=0\qquad \mbox{for $x\in \sigma_F\setminus (j+\sigma_F)$ and $j=1,\cdots, w_F$.}
\end{align*}
Taking into account the definition of $\sigma_F$, it is not difficult to see that this boundary conditions implies that for $n\in \sigma _F$ and $n+j\not \in \sigma_F$ then $h_j(n)=0$. Since $h_j(n)$ is a polynomial in $n$, we get from the definitions of $A_j^{a;F}$ (\ref{defaj}) and $\zeta _n$ (\ref{zi}) that $A_j^{a;F}$ is a rational function in $n$ whose denominator does not vanish for $n\in \NN$. Moreover, when $n\in \sigma _F$ and $n+j\not \in \sigma_F$ we deduce
$$
h_j(n)q_m^{a;F}(n+j)=\frac{h_j(n)\zeta_{n+j}c_{n+j}^{a;F}(m)}{\zeta_n}
$$
because $h_j(n)=0$ and $c_{n+j}^{a;F}=0$ as well. Hence, we can also apply duality in (\ref{expexpt}) to get (\ref{expexps}) and proceed as before.

\end{proof}

Our conjecture is that the higher order recurrence relation in Corollary \ref{cor1} has minimal order. In other words, the minimal order
for the recurrence relations of the form (\ref{horr}) satisfied by the exceptional Charlier polynomials is $2w_F+1$.

The expression (\ref{expexp}) for the higher order difference operator $D_F$ makes difficult to find explicitly the coefficients of its expansion (\ref{expexp2}) in terms of the shift operators $\Sh_n$, $n\in \ZZ$. These coefficients are needed to find explicit expressions for the recurrence coefficients $(A_j)_{j=-r}^r$ in (\ref{horrchex}). We have not been able to find explicit formulas for them in terms of arbitraries $a$ and $F$, but the expression (\ref{expexp}) allow as to find such explicit formulas for small values of $w_F$. Here it is an example.

Consider $F=\{1,2\}$. We have $w_F=3$ and then according to the Corollary \ref{cor1}, the exceptional Charlier polynomials $(c_n^{a;F})_n$ satisfy a seven order recurrence relation. Using (\ref{expexp}), we can compute explicitly the seven order difference operator with respect to which the polynomials $(q_n^{a;F})_n$ are eigenfunctions and the corresponding eigenvalues.
In doing that we get
$$
\sum_{j=-w_F}^{w_F}h_j(x)q_n^{a;F}(x+j)=\lambda _F^a(n)q_n^{a;F}(x),
$$
where
$$
h_j(x)=\begin{cases}
-x(x-4)(x-5)/6, &\mbox{if $j=-3$},\\
x(x-3)(x-4)/2, &\mbox{if $j=-2$},\\
-x(x-3)(x+a-2)/2, &\mbox{if $j=-1$},\\
(1/3-2a)x+(a-1/2)x^2+x^3/6, &\mbox{if $j=0$},\\
-ax(x-1+a)/2, &\mbox{if $j=1$},\\
a^2x/2, &\mbox{if $j=2$},\\
-a^3/6, &\mbox{if $j=3$,}
\end{cases}
$$
and $\displaystyle \lambda _F^a(n)=\frac{n^3}{6}+\frac{(1-a)n^2}{2}
+\frac{(2-3a+3a^2)n}{6}-\frac{a^3}{6}$.

Using (\ref{zi}) and (\ref{defaj}), we can explicitly find the coefficients $A_j^{a;F}$, $j=-3,\cdots , 3$, in Corollary \ref{cor1}:
\begin{equation}\label{coef1}
A_j^{a;F}(n)=\begin{cases}
a^3/6, &\mbox{if $j=-3$},\\
a^2(n/2-1), &\mbox{if $j=-2$},\\
(n-1)(a^2+(n-2)a)/2, &\mbox{if $j=-1$},\\
n^3/6+n^2(a-1/2)+n(1/3-2a), &\mbox{if $j=0$},\\
(n+1)(n-2)(a+n-1)/2, &\mbox{if $j=1$},\\
(n-1)(n^2-4)/2, &\mbox{if $j=2$},\\
(n+3)(n-1)(n-2)/6, &\mbox{if $j=3$,}
\end{cases}
\end{equation}
and, again, $\displaystyle \lambda_F^{a}(x)=\frac{x^3}{6}+\frac{(1-a)x^2}{2}
+\frac{(2-3a+3a^2)x}{6}-\frac{a^3}{6}$.

In this case, we have checked by using Maple that seven is the minimal order for a higher order recurrence relation for this family of exceptional Charlier polynomials (only linear equations are needed).

\bigskip

For every polynomial $\lambda$ such that $\lambda(x)-\lambda(x-1)$ is divisible by $\Omega _F^a$,
Theorem 3.2 in \cite{DdI} provides a higher order difference operator with respect to which the polynomials $q_n^{a;F}$, $n\in \NN$, are eigenfunctions with eigenvalues given by $\lambda(n)$. Again this difference operator is explicitly constructed.
For the example we are considering ($F=\{1,2\}$), we can then find other higher order recurrence relations for the exceptional Charlier polynomials which are not produced by iterating  that with coefficients (\ref{coef1}).
For instance, consider the polynomial
\begin{align*}
\tilde \lambda _F^a(x)=\frac{x^4}{8}+&\left(\frac{5}{12}-\frac{a}{2}\right)x^3
+\left(\frac{3a^2}{4}-a+\frac{3}{8}\right)x^2\\&+
\left(-\frac{a^3}{2}+\frac{3a^2}{4}-\frac{a}{2}+\frac{1}{12}\right)x+\frac{a^4}{8}-\frac{a^3}{6},
\end{align*}
which satisfies that
$$
\tilde \lambda _F^a(x)-\tilde \lambda _F^a(x-1)=c_1^a(x)\Omega _F^a(x),
$$
where $c_1^a(x)=x-a$ is the  Charlier polynomial of degree 1.

Proceeding as before, we find a nine order recurrence relation for this family of exceptional Charlier polynomials
of the form (\ref{horr}) where
\begin{equation}\label{coef2}
a_j(n)=\begin{cases}
a^4/8, &\mbox{if $j=-4$},\\
a^3(3n-8)/6, &\mbox{if $j=-3$},\\
a^2(n-2)(3n+2a-7))/4, &\mbox{if $j=-2$},\\
a(n-1)(n^2+3an-4n-7a+4)/2, &\mbox{if $j=-1$},\\
n^4/8+(3a/2-7/12)n^3+(3a^2/4-11a/2+7/8)n^2&\\\qquad+(-7a^2/4+5a-5/12)n, &\mbox{if $j=0$},\\
(3an-4a+n^2-2n+1)(n+1)(n-2)/2, &\mbox{if $j=1$},\\
(n-1)(n^2-4)(2a+3n-1)/4, &\mbox{if $j=2$},\\
(n+3)(n-1)(n-2)(1+3n)/6, &\mbox{if $j=3$},\\
(n+4)(n^2-1)(n-2)/8, &\mbox{if $j=4$},
\end{cases}
\end{equation}
and $\lambda(x)=\tilde \lambda _F^a(x)$.

\section{Exceptional Hermite polynomials}
We write $(H_n)_n$ for the sequence of Hermite polynomials (the next formulas can be found in \cite{Ch}, Ch. V; see also \cite{KLS}, pp, 250-3) defined by
\begin{equation}\label{Hpol}
 H_n(x)=n!\sum_{j=0}^{[n/2]}\frac{(-1)^j(2x)^{n-2j}}{j!(n-2j)!}.
\end{equation}
The Hermite polynomials are orthogonal with respect to the weight function $e^{-x^2}$, $x\in \RR$.

One can obtain Hermite polynomials from Charlier polynomials using the limit
\begin{equation}\label{blchh}
\lim_{a\to \infty}\left(\frac{2}{a}\right)^{n/2}c_n^a(\sqrt {2a}x+a)=\frac{1}{n!}H_n(x),
\end{equation}
see \cite{KLS}, p. 249.

Exceptional Hermite polynomials can be defined by means of the Wronskian
\begin{equation}\label{defhex}
H_n^F(x)=\left|
  \begin{array}{@{}c@{}cccc@{}c@{}}
    & H_{n-u_F}(x)&H_{n-u_F}'(x)&\cdots &H_{n-u_F}^{(k)}(x)& \\
    \dosfilas{H_{f}(x)&H_{f}'(x)&\cdots &H_{f}^{(k)}(x)}{f\in F}
  \end{array}
  \right|.
\end{equation}
As before, we have that for $n\in \sigma _F$ (see  (\ref{defsf})), $H_n^F$ is a polynomial of degree $n$.  But for $n\not \in \sigma_F$ the determinant (\ref{defhex}) vanishes and then $H_n^{F}=0$.

If we write
\begin{equation}\label{omh}
\Omega_F(x)=\det (H_{f_i}^{(j-1)}(x))_{i,j=1}^k,
\end{equation}
the admissibility condition (\ref{admch}) is equivalent
to the fact that the polynomial $\Omega_F(x)$ does not vanish in the real line (see \cite{Ad}, \cite{Kr} or \cite{duch}). Then, the polynomials $H_n^F$ are orthogonal with respect to the
positive weight
$$
\omega_{F}(x) =\frac{e^{-x^2}}{\Omega ^2_F(x)},\quad x\in \RR ,
$$
and they are called exceptional Hermite polynomials (see \cite{duch} and \cite{GUGM}).

Exceptional Hermite polynomials can be obtained from exceptional Charlier polynomials by using the basic limit (\ref{blchh}) (see \cite{duch}).
More precisely
\begin{equation}\label{lim1}
\lim_{a\to +\infty}\left(\frac{2}{a}\right)^{n/2}c_n^{a;F}(\sqrt{2a}x+a)=\frac{1}{(n-u_F)!\nu_F}H_n^F(x)
\end{equation}
uniformly in compact sets, where $\nu_F=2^{\binom{k+1}{2}}\prod_{f\in F}f!$

Up to an additive constant, we define the polynomial $\lambda_F$ of degree $w_F$ as the solution of the first order differential equation
\begin{equation}\label{mm1}
\lambda _F '=\frac{2^{k+1}}{\nu_F}\Omega _F,
\end{equation}
where $\Omega_F$ is the Wronskian (\ref{omh}).
We are now ready to establish the main result of this section.

\begin{corollary}\label{cor2} The exceptional Hermite polynomials defined by (\ref{defhex}) satisfy a $2w_F+1$ order recurrence relation of the form
\begin{equation}\label{horrhex}
\sum_{j=-w_F}^{w_F}A_j^{F}(n)H_{n+j}^{F}(x)=\lambda_F (x) H_{n}^{F}(x),\qquad n\ge 0,
\end{equation}
where the number $w_F$ and the polynomial $\lambda_F$ are defined by (\ref{defwf}) and (\ref{mm1}), respectively. For $j=-w_F,\cdots , w_F$, $A_j^{F}(n)$ is a rational function in $n$ which does not depend on $x$ and whose denominator does not vanish for $n\in \NN$.
\end{corollary}

\begin{proof}
The proof follows just by taking limit in Corollary \ref{cor1} as in \cite{duch}, Sections 5 and 6. In particular, one gets
$$
\lim_{a\to +\infty}\left(\frac{2}{a}\right)^{(u_F+k+1)/2}\lambda _F^a (\sqrt{2a}x+a)=\lambda _F (x)
$$
where $\lambda_F$ is defined by (\ref{mm1}).

\end{proof}

As for exceptional Charlier polynomials, our conjecture is that the minimal order for the recurrence relations
of the form (\ref{horr}) satisfied by the exceptional Hermite  polynomials is
$2w_F+1$ and it corresponds with the recurrence relation given in the previous corollary.

\bigskip
For $F=\{ 1,2\}$, taking limit in (\ref{coef1}) and (\ref{coef2}), we get explicit expressions for the recurrence coefficients of
a seven and a nine term recurrence relation for the exceptional Hermite polynomials $(H_n^F)_n$. More precisely, the exceptional Hermite polynomials $(H_n^F)_n$
satisfy the seven order recurrence relation (\ref{horrhex}) where $w_F=3$,
$$
A_j^F(n)=\begin{cases}
4n(n-1)(n-2)/3, &\mbox{if $j=-3$},\\
2n(n-1), &\mbox{if $j=-1$},\\
n-2, &\mbox{if $j=1$},\\
\frac{(n-1)(n-2)}{6(n+1)(n+2)}, &\mbox{if $j=3$},\\
0, &\mbox{if $j=-2,0,2$,}
\end{cases}
$$
and $\lambda_F(x)=4x^3/3+2x$.

They also satisfy the nine order recurrence relation (\ref{horr}) where $r=4$,
$$
a_j(n)=\begin{cases}
2n(n-1)(n-2)(n-3), &\mbox{if $j=-4$},\\
4n(n-1)(n-2), &\mbox{if $j=-2$},\\
n(3n-7), &\mbox{if $j=0$},\\
\frac{(n-1)(n-2)}{n+1}, &\mbox{if $j=2$},\\
\frac{(n-1)(n-2)}{8(n+2)(n+3)}, &\mbox{if $j=4$},\\
0, &\mbox{if $j=-3,-1,1,3$},
\end{cases}
$$
and $\lambda(x)=2x^4+2x^2-1/2$.

\section{Exceptional Meixner polynomials}
We start with some basic definitions and facts about Meixner  polynomials.

For $a\not =0, 1$ we write $(m_{n}^{a,c})_n$ for the sequence of Meixner polynomials defined by
\begin{equation}\label{Mxpol}
m_{n}^{a,c}(x)=\frac{a^n}{(1-a)^n}\sum _{j=0}^n a^{-j}\binom{x}{j}\binom{-x-c}{n-j}
\end{equation}
(we have taken a slightly different normalization from the one used in \cite{Ch}, pp. 175-7, from where
the next formulas can be easily derived; see also \cite{KLS}, pp, 234-7 or \cite{NSU}, ch. 2).
Meixner polynomials are eigenfunctions of the following second order difference operator
\begin{equation}\label{Mxdeq}
D_{a,c} =\frac{x\Sh_{-1}-[(1+a)x+ac]\Sh_0+a(x+c)\Sh_1}{a-1},\qquad D_{a,c} (m_{n}^{a,c})=nm_{n}^{a,c},\quad n\ge 0,
\end{equation}
where $\Sh_l$ denotes the shift operator $\Sh_l(f)=f(x+l)$.
For $a\not =0,1$ and $c\not =0,-1,-2,\ldots $, they are always orthogonal with respect to a moment functional $\rho_{a,c}$. For $0<\vert a\vert<1$ and $c\not =0,-1,-2,\ldots $, we have
\begin{equation*}\label{MXw}
\rho_{a,c}=\sum _{x=0}^\infty \frac{a^{x}\Gamma(x+c)}{x!}\delta _x.
\end{equation*}
The moment functional $\rho_{a,c}$ can be represented by a positive measure only when $0<a<1$ and $c>0$.

From now on, $\F=(F_1,F_2)$ will denote a pair of finite sets of
positive integers. We denote by $k_j$ the number of elements of $F_j$,
$j=1,2$, and $k=k_1+k_2$ is the number of elements of $\F$. One of
the components of $\F$, but not both, can be the empty set.

We associate to $\F$ the nonnegative integers $u_\F$ and $w_\F$ and the infinite set of nonnegative integers $\sigma_\F$ defined by
\begin{align}\label{defuf2}
u_\F&=\sum_{f\in F_1}f+\sum_{f\in
F_2}f-\binom{k_1+1}{2}-\binom{k_2}{2},\\\label{defwf2}
w_\F&=\sum_{f\in F_1}f+\sum_{f\in
F_2}f-\binom{k_1}{2}-\binom{k_2}{2}+1,\\\label{defsf2}
\sigma _\F&=\{u_\F,u_\F+1,u_\F+2,\cdots \}\setminus \{u_\F+f,f\in
F_1\}.
\end{align}
The infinite set $\sigma_\F$ will be the set of indices for the exceptional Meixner or Laguerre polynomials associated to $\F$.

We are now ready to introduce exceptional Meixner  polynomials (see \cite{dume}).

\begin{definition}
Let $\F =(F_1,F_2)$ be a pair of finite sets of positive integers. For real numbers $a,c$, with $a\not = 0,1$ and $c\not =0,-1,-2,\ldots$, we define the polynomials $m_n^{a,c;\F}$, $n\ge 0$, as
\begin{equation}\label{defmex}
m_n^{a,c;\F}(x)=  \left|
  \begin{array}{@{}c@{}cccc@{}c@{}}
    & m_{n-u_\F}^{a,c}(x)&m_{n-u_\F}^{a,c}(x+1)&\cdots &m_{n-u_\F}^{a,c}(x+k) & \\
    \dosfilas{ m_{f}^{a,c}(x) & m_{f}^{a,c}(x+1) &\cdots  & m_{f}^{a,c}(x+k) }{f\in F_1} \\
    \dosfilas{ m_{f}^{1/a,c}(x) & m_{f}^{1/a,c}(x+1)/a & \cdots & m_{f}^{1/a,c}(x+k)/a^k }{f\in F_2}
  \end{array}
  \right|
\end{equation}
where the number $u_\F$ is defined by (\ref{defuf2}).
\end{definition}
The determinant (\ref{defmex}) should be understood as explained in   (\ref{defdosf}).

As before, we have that for $n\in \sigma _\F$ (see  (\ref{defsf2})), $m_n^{a,c;\F}$ is a polynomial of degree $n$.  But for $n\not \in \sigma_\F$ the determinant (\ref{defmex}) vanishes and then $m_n^{a,c;\F}=0$.

In \cite{dume}, we proved that these polynomials are always eigenfunctions of a second order difference operator with rational coefficients. Under the assumption $0<a<1$, $c\not =0,-1,-2,\cdots $ and the admissibility condition
\begin{equation}\label{defadmi}
\frac{\prod_{f\in F_1}(x-f)\prod_{f\in F_2}(x+c+f)}{(x+c)_{\hat c}}\ge 0,\quad x\ge 0,
\end{equation}
where $\hat c=\max \{-[c],0\}$ and $[c]$ denotes the value of the floor function at $c$ (i.e. $[c]=\max\{s\in \ZZ: s\le c\}$),
the polynomials $(m_n^{a,c;\F})_{n\in\sigma_\F}$ are orthogonal (and complete ) with respect to a positive measure (see Theorems 4.3 and 4.4 in \cite{dume}). We call these polynomials exceptional Meixner polynomials. Since we want to work with orthogonal polynomials with respect to positive measure we will assume that $0<a<1$, $c\not =0,-1,-2,\cdots $ and that the admissibility condition (\ref{defadmi}) holds, although these assumptions are not need for the implementation of our method to find higher order recurrence relations for the polynomials (\ref{defmex}).

Related to the exceptional Meixner polynomials is the Casoratian type determinant defined by
\begin{equation}\label{defom}
\Omega _\F^{a,c}(x)=  \left|
  \begin{array}{@{}c@{}cccc@{}c@{}}
    \dosfilas{ m_{f}^{a,c}(x) & m_{f}^{a,c}(x+1) &\cdots  & m_{f}^{a,c}(x+k-1) }{f\in F_1} \\
    \dosfilas{ m_{f}^{1/a,c}(x) & m_{f}^{1/a,c}(x+1)/a & \cdots & m_{f}^{1/a,c}(x+k-1)/a^{k-1} }{f\in F_2}
  \end{array}
  \right|.
\end{equation}
$\Omega _\F^{a,c}$ is a polynomial of degree $w_\F-1$ (see (\ref{defwf2}) for the definition of $w_\F$). In \cite{dume}, it has been conjectured that this determinant enjoys the symmetry
\begin{equation}\label{iza}
\Omega_\F^{a,c}(x)=(-1)^{u_\F+k_1}\frac{u_a(\F)}{u_a(\G)}\Omega_\G^{a,-c-M_{F_1}-M_{F_2}}
(-x),
\end{equation}
where $u_a(\F)=a^{\binom{k_2}{2}-k_2(k-1)}(1-a)^{k_1k_2}$.

Up to an additive constant, we define the polynomial $\lambda_\F^{a,c}$ of degree $w_\F$ as the solution of the first order difference equation
\begin{equation}\label{lme}
\lambda_\F^{a,c}(x)-\lambda_\F^{a,c}(x-1)=\Omega_\G^{a,-c-\max F_1-\max F_2}(-x),
\end{equation}
where $\G=(I(F1),I(F_2))$, and $I$ is the involution defined by (\ref{dinv}).
As we will see below, the higher order recurrence relation for the exceptional Meixner polynomials is constructed from this polynomial $\lambda_\F^{a,c}$.

Consider now the measure
\begin{equation}\label{mrafm}
\rho _{a,c}^{\F}=\sum _{x=u_\F}^\infty \prod_{f\in F_1}(x-f-u_\F)\prod_{f\in F_2}(x+c+f-u_\F)\frac{a^{x-u_\F}\Gamma(x+c-u_\F)}{(x-u_\F)!}\delta _x.
\end{equation}
For $0<a<1$ and $c\not =0,-1,-2,\cdots$ and under the assumption (\ref{defadmi}), this measure is positive. Notice that this measure is supported in the infinite set of nonnegative integers $\sigma_\F$ (\ref{defsf2}).

The measure $\rho _{a,c}^{\F}$ has associated a sequence of orthogonal polynomials $q_n^{a,c;\F}$, $n\ge 0$, which can be constructed using
the Christoffel-Szeg\"o determinantal formula
\begin{equation}\label{defqnme}
q_n^{a,c;\F}(x)=\frac{\left|
  \begin{array}{@{}c@{}cccc@{}c@{}}
    & m_n^{a,c}(x-u_\F)&m_{n+1}^{a,c}(x-u_\F)&\cdots &m_{n+k}^{a,c}(x-u_\F) & \\
    \dosfilas{ m_{n}^{a,c}(f) & m_{n+1}^{a,c}(f) &\cdots  & m_{n+k}^{a,c}(f) }{f\in F_1} \\
    \dosfilas{m_{n}^{1/a,c}(f) & -m_{n+1}^{1/a,c}(f) & \cdots & (-1)^{k}m_{n+k}^{1/a,c}(f)}{f\in F_2}
  \end{array}
  \right|}{(-1)^{nk_2}\prod_{f\in F_1}(x-f-u_\F)\prod_{f\in F_2}(x+c+f-u_\F)}.
\end{equation}
In \cite{DdI}, Theorem 6.2, it is proved that the polynomials $q_n^{a,c;\F}(x+u_\F)$, $n\ge 0$, are eigenfunctions of a higher order difference operator. This operator can be explicitly constructed by means of the formula
\begin{align}\label{expexp3}
D_\F=\lambda_F^{a,c}(D_{a,c})&+\sum_{g\in I(F_1)}\frac{M_g^1(D_{a,c})}{1-a}\circ \nabla \circ m_g^{a,-c-\max I(F_1)-\max I(F_2)}(-D_{a,c}-1)\\\nonumber
&+\sum_{g\in I(F_2)}\frac{a M_g^2(D_{a,c})}{1-a}\circ \Delta \circ m_g^{1/a,-c-\max I(F_1)-\max I(F_2)}(-D_{a,c}-1)
\end{align}
where $D_{a,c}$ is the Meixner second order difference operator (\ref{Mxdeq}), $\lambda _\F^{a,c}$ is the polynomial (\ref{lme}) and $M_{g}^i$, $g\in I(F_i)$, $i=1,2$, are certain polynomials with can be explicitly constructed (see \cite{dume}, Section 6). The associated eigenvalues are given
by the polynomial $\lambda _\F^{a,c}$, so that $D_\F(q_n^{a,c;\F})=\lambda _\F^{a,c}(n)q_n^{a,c;\F}$.

In \cite{DdI}, Theorem 6.2, it is also proved that $D_\F$ is a difference operator of order $2w_\F+1$, which can be written in terms of the shift operators $\Sh_j$ in the form
\begin{equation}\label{expexp4}
D_\F=\sum_{j=-w_\F}^{w_\F} h_j(x)\Sh _j,
\end{equation}
where $h_j$, $j=-w_F,\cdots,w_F$, are certain polynomials.

It turns out that the exceptional Meixner polynomials $m_n^{a,c;\F}$, $n\in \sigma_\F$, are strongly related by duality with  the polynomials $q_n^{a,c;\F}$, $n\ge 0$.

\begin{lemma}[Lemma 3.2 of \cite{dume}]\label{lem3.3}
If $u$ is a nonnegative integer and $v\in \sigma_\F$, then
\begin{equation}\label{duaqnrn2}
q_u^\F(v)=\kappa\xi_u\zeta_vm_v^\F(u),
\end{equation}
where
\begin{align}\nonumber
\kappa&=\frac{(-1)^{\sum _{f\in F_2}f}a^{k_2(k_1+1)+\sum _{f\in F_2}f}\prod_{f\in F_1}f!\prod_{f\in F_2}f!}{(a-1)^{k_2(k_1+1)}\prod_{f\in F_1}(1+c)_{f-1}\prod_{f\in F_2}(1+c)_{f-1}},\\\nonumber
\xi_u&=\frac{a^{(k_1+1)u}\prod_{i=0}^k(1+c)_{u+i-1}}
{(a-1)^{(k+1)u}\prod_{i=0}^k(u+i)!},\\\label{asqs}
\zeta_v&=\frac{(a-1)^{v}(v-u_\F)!}{a^v(1+c)_{v-u_\F-1}\prod_{f\in F_1}(v-f-u_\F)\prod_{f\in F_2}(v+c+f-u_\F)}.
\end{align}
\end{lemma}
We are now ready to establish the main result of this section.

\begin{corollary}\label{cor3} The exceptional Meixner polynomials defined by (\ref{defmex}) satisfy a $2w_\F+1$ order recurrence relation of the form
\begin{equation}\label{horrmex}
\sum_{j=-w_\F}^{w_\F}A_j^{a,c;\F}(n)m_{n+j}^{a,c;\F}(x)=\lambda _\F^{a,c} (x) m_{n}^{a,c;\F}(x),\quad n\ge 0,
\end{equation}
where the number $w_\F$ and the polynomial $\lambda_\F^{a,c}$ are defined by (\ref{defwf2}) and (\ref{lme}), respectively. For $j=-w_\F,\cdots , w_\F$, $A_j^{a,c;\F}(n)$ is a rational function in $n$ which does not depend on $x$ and whose denominator does not vanish in $n\in \NN$.
\end{corollary}

\begin{proof}
The proof is similar to that of Corollary \ref{cor1} and it is omitted. We only point out that the relationship between the coefficients $A_j^{a,c;\F}$ (\ref{horrmex}) and $h_j$ (\ref{expexp4}) is given by
\begin{equation}\label{esum}
A_j^{a,c;\F}(n)=\frac{h_j(n)\zeta_{n+j}}{\zeta_n},
\end{equation}
where $\zeta_n$ is defined by (\ref{asqs})

\end{proof}

Our conjecture is that the minimal order for the higher order recurrence relations of the form (\ref{horr}) satisfied by the exceptional Meixner polynomials is $2w_\F+1$.

Here it is a trio of examples.

Consider $F_1=\{1,2\}, F_2=\emptyset$ and $\F=(F_1,F_2)$.
We have $w_\F=3$ and then according to the Corollary \ref{cor3}, the exceptional Meixner polynomials $(m_n^{a,c;\F})_n$ satisfy a seven order recurrence relation. Using (\ref{expexp3}), we can compute explicitly the seven order difference operator with respect to which the polynomials $(q_n^{a,c;\F})_n$ are eigenfunctions and the corresponding eigenvalues. Applying then  (\ref{asqs})
and (\ref{esum}) we can explicitly find the coefficients $A_j^{a,c;\F}$, $j=-3,\cdots , 3$:
\begin{equation}\label{coefm1}
A_j^{a,c;\F}(n)=\begin{cases}
\frac{a^3(n+c-3)(n+c-2)(n+c-1)}{6(a-1)^6}, &\mbox{if $j=-3$},\\
-\frac{a^2(a+1)(n+c-2)(n+c-1)(n-2)}{2(a-1)^5}, &\mbox{if $j=-2$},\\
\frac{a(n+c-1)(n-1)((n-2)(a^2+3a+1)+ac)}{2(a-1)^4}, &\mbox{if $j=-1$},\\
-\frac{(a+1)[(a^2+8a+1)n(n-1)(n-2)/6+acn(n-2)]}{(a-1)^3}&\\\qquad\frac{-a^3c(c+1)(c+2)}{6(a-1)^3}, &\mbox{if $j=0$},\\
\frac{(n+1)(n-2)((n-1)(a^2+3a+1)+ac)}{(a-1)^2}, &\mbox{if $j=1$},\\
-\frac{(a+1)(n-1)(n^2-4)}{2(a-1)}, &\mbox{if $j=2$},\\
\frac{(n+3)(n-1)(n-2)}{6}, &\mbox{if $j=3$}
\end{cases}
\end{equation}
and $\displaystyle \lambda_\F^{a,c}(x)=\frac{x^3}{6}+\frac{(a+ac-1)x^2}{2(a-1)}
+\frac{(3ac(2a+ac-1)+2(a-1)^2)x}{6(a-1)^2}$.

In this case, we have checked by using Maple that seven is the minimal order for a higher order recurrence relation for this family of exceptional Meixner polynomials (only linear equations are needed).

The admissibility condition (\ref{defadmi}) for this example reduces to $c\in (-2,-1)\cup (0,+\infty)$, but the recurrence formula also holds for
$a\not =0,1$ and $c\not =0,-1,-2,\cdots$.

\bigskip

Consider $F_1=\emptyset, F_2=\{1\}$. We have $w_F=2$ and then according to the Corollary \ref{cor3}, the exceptional Meixner polynomials $(m_n^{a,c;\F})_n$ satisfy a five order recurrence relation. Proceeding as before, we get:
\begin{equation}\label{coefm2}
A_j^{a,c;\F}(n)=\begin{cases}
-\frac{a^2(n+c-3)(n+c)}{2(a-1)^4}, &\mbox{if $j=-2$},\\
\frac{a(a+1)(n+c-2)(n+c))}{(a-1)^3}, &\mbox{if $j=-1$},\\
-\frac{(a^2/2+2a+1/2)n(n-1)-c(a^2+3a+1)n}{(a-1)^2}& \\\qquad-\frac{c(c(a^2+2a)-a^2-2a-2)}{2(a-1)^2}, &\mbox{if $j=0$},\\
\frac{(a+1)n(n+c)}{a-1}, &\mbox{if $j=1$},\\
-\frac{n(n+1)}{2}, &\mbox{if $j=2$}
\end{cases}
\end{equation}
and $\displaystyle \lambda_F^{a}(x)=-\frac{x(x(a-1)+a-2c-1)}{2(a-1)}$.

In this case, we have checked by using Maple that five is the minimal order for a higher order recurrence relation for this family of exceptional Meixner polynomials (only linear equations are needed).

Consider $F_1=\{1\}, F_2=\{1\}$.
We have $w_\F=3$ and then according to the Corollary \ref{cor3}, the exceptional Meixner polynomials $(m_n^{a,c;\F})_n$ satisfy a seven order recurrence relation. Using (\ref{expexp2}), we can compute explicitly the seven order difference operator with respect to which the polynomials $(q_n^{a,c;\F})_n$ are eigenfunctions and the corresponding eigenvalues. Applying then Corollary (\ref{cor3})
and (\ref{esum}) we can explicitly find the coefficients $A_j^{a,c;\F}$, $j=-3,\cdots , 3$:
\begin{equation}\label{coefm3}
A_j^{a,c;\F}(n)=\begin{cases}
-\frac{a^2(n+c-4)(n+c-2)(n+c)}{3(a-1)^5}, &\mbox{if $j=-3$},\\
\frac{a(a+1)(n+c-3)(n+c)(2n+c-4)}{2(a-1)^4}, &\mbox{if $j=-2$},\\
-\frac{(a^2+3a+1)(n+c-2)(n+c)(n-2)}{(a-1)^3}, &\mbox{if $j=-1$},\\
\frac{(a+1)n[(a^2+8a+1)(2n^2+3(c-2)n+4)-3c(3a^2+(-2c+20)a+3)]}{6a(a-1)^2}&\\\qquad -\frac{c(a^3(c+4)(c-1)+3a^2(c+8)(c-1)+6a(c-7)-6)}{6a(a-1)^2}, &\mbox{if $j=0$},\\
-\frac{(a^2+3a+1)(n+c)(n-2)n}{a(a-1)}, &\mbox{if $j=1$},\\
\frac{(a+1)(n-2)(n+1)(2n+c)}{2a}, &\mbox{if $j=2$},\\
-\frac{(a-1)n(n^2-4)}{3a}, &\mbox{if $j=3$}
\end{cases}
\end{equation}
and $\displaystyle \lambda_F^{a}(x)=-\frac{(a-1)x^3}{3a}-\frac{(2a+ac-2-c)x^2}{2a}
-\frac{(-6c^2a+3(a^2-6a+1)c+4(a-1)^2)x}{6a(a-1)}$.

\section{Exceptional Laguerre polynomials}
For $\alpha\in\mathbb{R}$, we write $(L_n^\alpha )_n$ for the sequence of Laguerre polynomials
\begin{equation}\label{deflap}
L_n^{\alpha}(x)=\sum_{j=0}^n\frac{(-x)^j}{j!}\binom{n+\alpha}{n-j}
\end{equation}
(that and the next formulas can be found in \cite{EMOT}, vol. II, pp. 188--192; see also \cite{KLS}, pp, 241-244).

For $\alpha\neq-1,-2,\ldots$, they are orthogonal with respect to a measure $\mu_{\alpha}=\mu_{\alpha}(x)dx$. This measure is positive
only when $\alpha>-1$ and then
$$
\mu_{\alpha}(x) =x^\alpha e^{-x}, x>0.
$$
One can obtain Laguerre polynomials from Meixner polynomials using the limit
\begin{equation}\label{blmel}
\lim_{a\to 1}(a-1)^nm_n^{a,c}\left(\frac{x}{1-a}\right)=L_n^{c-1}(x)
\end{equation}
see \cite{KLS}, p. 243.

Exceptional Laguerre polynomials can be defined by means of the Wronskian type determinant
\begin{equation}\label{deflax}
L_n^{\alpha ;\F}(x)= \left|
  \begin{array}{@{}c@{}cccc@{}c@{}}
    & L_{n-u_\F}^{\alpha}(x)&(L_{n-u_\F}^{\alpha})'(x)&\cdots &(L_{n-u_\F}^{\alpha})^{(k)}(x) & \\
    \dosfilas{ L_{f}^{\alpha}(x) & (L_{f}^{\alpha})'(x) &\cdots  & (L_{f}^{\alpha})^{(k)}(x) }{f\in F_1} \\
    \dosfilas{ L_{f}^{\alpha}(-x) & L_{f}^{\alpha+1}(-x) & \cdots & L_{f}^{\alpha +k}(-x) }{f\in F_2}
  \end{array}
  \right|.
\end{equation}
As before, we have that for $n\in \sigma _\F$ (see  (\ref{defsf2})), $L_n^{\alpha ;\F}$ is a polynomial of degree $n$.  But for $n\not \in \sigma_\F$ the determinant (\ref{defmex}) vanishes and then $L_n^{\alpha ;\F}=0$.

If we write
\begin{equation}\label{defhom}
\Omega _{\alpha;\F}^{\alpha}(x)=\left|
  \begin{array}{@{}c@{}cccc@{}c@{}}
    \dosfilas{ L_{f}^{\alpha}(x) & (L_{f}^{\alpha})'(x) &\cdots  & (L_{f}^{\alpha})^{(k-1)}(x) }{f\in F_1} \\
    \dosfilas{ L_{f}^{\alpha}(-x) & L_{f}^{\alpha+1}(-x) & \cdots & L_{f}^{\alpha +k-1}(-x) }{f\in F_2}
  \end{array}
  \right| ,
\end{equation}
the admissibility condition (\ref{defadmi}) is equivalent to the fact that the polynomial $\Omega _{\alpha;\F}$ does not vanish in $[0,+\infty)$ (see \cite{dume} and \cite{dupe}). Then, the polynomials $L_n^{\alpha;\F}$ are orthogonal with respect to the positive weight
$$
\omega_{\alpha;\F}(x)=\frac{x^{\alpha +k}e^{-x}}{(\Omega_\F^\alpha(x))^2},\quad x>0,
$$
and they are called exceptional Laguerre polynomials (see \cite{dume}).

Exceptional Laguerre polynomials can be obtained from exceptional Meixner polynomials by using the basic limit (\ref{blmel}). More precisely
\begin{equation}\label{lim2}
\lim_{a\to 1}(a-1)^{n-(k_1+1)k_2}m_n^{a,c;\F}\left(\frac{x}{1-a}\right)=(-1)^{\binom{k+1}{2}+\sum_{f\in F_2}f}L_n^{\alpha ;\F}(x)
\end{equation}
uniformly in compact sets.

Up to an additive constant, we define the polynomial $\lambda_{\alpha;\F}$ of degree $w_\F$ as the solution of the first order differential equation
\begin{equation}\label{mm2}
\lambda _{\alpha;\F}'(x)=\Omega _\G^{-\alpha-\max F_1- \max F_2-2}(-x),
\end{equation}
where $\G=(I(F_1),I(F_2))$, $I$ is the involution defined by (\ref{dinv}) and $\Omega_\F^\alpha $ is the Wronskian type determinant (\ref{defhom}).
We are now ready to establish the main result of this section (the proof is omitted because is similar to that of \ref{cor2}).

\begin{corollary}\label{cor4} The exceptional Laguerre polynomials defined by (\ref{deflax}) satisfy a $2w_\F+1$ order recurrence relation of the form
\begin{equation}\label{horrlax}
\sum_{j=-w_\F}^{w_\F}A_j^{\alpha;\F}(n)L_{n+j}^{\alpha;\F}(x)=\lambda_{\alpha; \F} (x) L_{n}^{\alpha;\F}(x),
\end{equation}
where the number $w_\F$ and the polynomial $\lambda_{\alpha ;\F}$ are defined by (\ref{defwf2}) and (\ref{mm2}), respectively. For $j=-w_\F,\cdots , w_\F$, $A_j^{\alpha;\F}(n)$ is a rational function in $n$ which does not depend on $x$ and whose denominator does not vanish for $n\in \NN$.
\end{corollary}

As for the other families of exceptional polynomials considered in this paper, our conjecture is that the minimal order for the higher order recurrence relations of the form (\ref{horr}) satisfied by the exceptional Laguerre polynomials is $2w_\F+1$.

For $F_1=\emptyset$ and $F_2=\{\ell\}$, our family $L_n^{\alpha;\F}$ coincides, up to renormalization, with the so-called type I exceptional Laguerre polynomials. In \cite{STZ}, the authors prove that type I exceptional Laguerre polynomials satisfy a $4\ell+1$ order recurrence relation of the form (\ref{horr}) (an explicit expression for this recurrence relation is not provided in \cite{STZ}). Since for this particular pair $\F=(F_1,F_2)$, $w_\F=\ell+1$, we have that our corollary gives a $2\ell+3$ order recurrence relation for this family.

\bigskip
For $\F=(F_1,F_2)$, with $F_1=\{ 1,2\}, F_2=\emptyset$, taking limit in (\ref{coefm1}), we get explicit expressions for the recurrence coefficients of
a seven term recurrence relation for the exceptional Laguerre polynomials $(L_n^{\alpha;\F})_n$. More precisely, the exceptional Laguerre polynomials $(L_n^{\alpha;\F})_n$ satisfy the seven order recurrence relation (\ref{horrlax}) where $w_\F=3$,
$$
A_j^{\alpha;\F}(n)=\begin{cases}
-(n+\alpha)(n+\alpha-1)(n+\alpha-2)/6, &\mbox{if $j=-3$},\\
(n+\alpha)(n+\alpha-1)(n-2), &\mbox{if $j=-2$},\\
-n+\alpha)(5n+\alpha-9)(n-1)/2, &\mbox{if $j=-1$},\\
10n^3/3-(8-2\alpha)n^2-(4\alpha-8/3)n&\\\qquad +\alpha^3/6+\alpha^2+11\alpha/6+1, &\mbox{if $j=0$},\\
-(5n+\alpha-4)(n+1)(n-2)/2, &\mbox{if $j=1$},\\
(n-1)(n^2-4), &\mbox{if $j=2$},\\
-(n+3)(n-1)(n-2)/6, &\mbox{if $j=3$}
\end{cases}
$$
and $\displaystyle \lambda_F^{\alpha}(x)=x(x^2-3x(a+1)+3(a+1)(a+2))/6$.

\bigskip
For $\F=(F_1,F_2)$, with $F_1=\emptyset, F_2=\{ 1\}$, taking limit in (\ref{coefm2}) , we find that the exceptional Laguerre polynomials $(L_n^{\alpha;\F})_n$ satisfy the five order recurrence relation (\ref{horrlax}) where $w_\F=2$,
$$
A_j^{\alpha;\F}(n)=\begin{cases}
-(n+\alpha+1)(n+\alpha-2)/2, &\mbox{if $j=-2$},\\
2(n+\alpha+1)(n+\alpha-1), &\mbox{if $j=-1$},\\
-3n^2-(2+5\alpha)n-3\alpha^2/2-\alpha/2+1, &\mbox{if $j=0$},\\
2n(n+\alpha+1), &\mbox{if $j=1$},\\
-n(n+1)/2, &\mbox{if $j=2$},
\end{cases}
$$
and $\displaystyle \lambda_F^{\alpha}(x)=-x(x+2\alpha+2)/2$.

\bigskip
\bigskip
For $\F=(F_1,F_2)$, with $F_1=\{ 1\}, F_2=\{ 1\}$, taking limit in (\ref{coefm3}) , we find that the exceptional Laguerre polynomials $(L_n^{\alpha;\F})_n$ satisfy the seven order recurrence relation (\ref{horrlax}) where $w_\F=3$,
$$
A_j^{\alpha ;\F}(n)=\begin{cases}
-(n+\alpha-3)(n+\alpha-1)(n+\alpha +1)/3, &\mbox{if $j=-3$},\\
(n+\alpha -2)(n+\alpha +1)(2n+\alpha-3), &\mbox{if $j=-2$},\\
-5(n+\alpha-1)(n+\alpha +1)(n-2), &\mbox{if $j=-1$},\\
-(2n+\alpha-1)(-10n^2-10(\alpha-1)n+2\alpha^2+23\alpha+21)/3, &\mbox{if $j=0$},\\
-5(n+\alpha+1)(n-2)n, &\mbox{if $j=1$},\\
(n-2)(n+1)(2n+\alpha+1), &\mbox{if $j=2$},\\
-n(n^2-4)/3, &\mbox{if $j=3$}
\end{cases}
$$
and $\displaystyle \lambda_F^{a}(x)=x(x^2-3(a+3)(a+1))/3$.

\bigskip
\noindent
\textit{Mathematics Subject Classification: 42C05, 33C45, 33E30}

\noindent
\textit{Key words and phrases}: Orthogonal polynomials. Exceptional orthogonal polynomial. Recurrence relations. Difference operators. Differential operators. Charlier polynomials. Meixner polynomials. Hermite polynomials. Laguerre polynomials.

     \end{document}